\documentclass[11pt]{amsart}

\usepackage{amsmath}
\usepackage{amssymb,amsthm,amsfonts}
\usepackage{amsrefs}
\usepackage[mathscr]{euscript}
\usepackage{a4wide}
\usepackage{graphicx}
\usepackage{wrapfig}
\usepackage[usenames,dvipsnames,svgnames,table]{xcolor}

\renewcommand{\mathcal}{\mathscr}

\def\R {\mathbb{R}}

\def\N {\mathbb{N}}

\def\Z {\mathbb{Z}}

\renewcommand{\epsilon}{\varepsilon}
\newcommand{\eps}{\varepsilon}

\renewcommand{\leq}{\leqslant}
\renewcommand{\le}{\leqslant}
\renewcommand{\geq}{\geqslant}
\renewcommand{\ge}{\geqslant}

\DeclareMathOperator*{\osc}{osc}
\newcommand{\Per}{\mathrm{Per}}

\newtheorem{proposition}{Proposition}[section]
\newtheorem{theorem}[proposition]{Theorem}
\newtheorem*{theorem*}{Theorem}

\newtheorem{lemma}[proposition]{Lemma}

\theoremstyle{definition}
\newtheorem{definition}[proposition]{Definition}

\newtheorem{remark}[proposition]{Remark}
\numberwithin{equation}{section}

\usepackage{calligra}
\usepackage[T1]{fontenc}

\begin{document}

\title{Minimizers of the $p$-oscillation functional}\thanks{This work
has been supported by the Australian Research Council grant
``N.E.W.'' {\em Nonlocal Equation at Work}.\\
{\em Annalisa Cesaroni}:
Dipartimento di Scienze Statistiche,
Universit\`a di Padova, Via Battisti 241/243, 35121 Padova, Italy. {\tt annalisa.cesaroni@unipd.it}
\\
{\em Serena Dipierro}:
Dipartimento di Matematica, Universit\`a di Milano,
Via Saldini 50, 20133 Milan, Italy.
{\tt serena.dipierro@unimi.it}\\
{\em Matteo Novaga}: Dipartimento di Matematica,
Universit\`a di Pisa, 
Largo Pontecorvo 5, 56127 Pisa,
Italy. {\tt matteo.novaga@unipi.it}\\
{\em Enrico Valdinoci}:
School of Mathematics
and Statistics,
University of Melbourne, 813 Swanston St, Parkville VIC 3010, Australia,
Dipartimento di Matematica, Universit\`a di Milano,
Via Saldini 50, 20133 Milan, Italy, and IMATI-CNR, Via Ferrata 1, 27100 Pavia,
Italy. {\tt enrico@mat.uniroma3.it}}
\author[A. Cesaroni]{Annalisa Cesaroni}
\author[S. Dipierro]{Serena Dipierro}
\author[M. Novaga]{Matteo Novaga}
\author[E. Valdinoci]{Enrico Valdinoci}

\subjclass[2010]{
35B05, 
49Q20, 
35B53
}
\keywords{$p$-oscillation functional, Minkowski content, discrete total variation functional.}

\begin{abstract}  
We define a family of functionals, called $p$-oscillation functionals, 
that  can be interpreted as discrete versions of the classical total variation functional for $p=1$ and of the 
$p$-Dirichlet functionals for $p>1$.  We introduce the notion of minimizers and prove existence of solutions to the Dirichlet problem. Finally
we provide a description of Class~A minimizers (i.e. minimizers under compact perturbations) in dimension $1$. 
\end{abstract}

\maketitle

\begin{center}
{\calligra\Large To Luis Caffarelli, on the occasion of his 70th birthday}
\end{center}\bigskip

\tableofcontents
\section{Introduction}

Given~$\Omega\subseteq\R^n$, the classical quadratic Dirichlet energy
\begin{equation} \label{DIRI:2}   \int_\Omega |\nabla u(x)|^2\,dx\end{equation}
and its generalization to any homogeneous energy of the form
\begin{equation} \label{DIRI:p}
\int_\Omega |\nabla u(x)|^p\,dx,\qquad p\in[1,+\infty)\end{equation}
constitute the foundation of the modern analysis and of the calculus of variations
(see e.g. the Introduction in~\cite{COURANT77} for a detailed historical overview).
In particular, the minimization of the functional in~\eqref{DIRI:2} with prescribed boundary data is related to harmonic functions, while the functional in~\eqref{DIRI:p} gives rise to the
$p$-Laplace operator. In general, the functionals in~\eqref{DIRI:2} and~\eqref{DIRI:p} are the main building blocks for a number of problems in elasticity, heat conduction, population dynamics, etc.

In the recent years, suitable generalizations of
the functionals in~\eqref{DIRI:2} and~\eqref{DIRI:p} have been taken into account in the literature,
with the aim of modeling situations in which different scales come into play. Besides the natural mathematical curiosity, this type of problems is motivated by several concrete applications in which the setting is not scale invariant: for instance, in the digitalization process of images with tiny details
(e.g. fingerprints, tissues, layers, etc.) the use of
different scales allows the preservation of fine structures, precise elements and irregularities of the image in the process of removing white noises, and this constitutes an essential ingredient in the process of improving the quality of the data without losing important information.\medskip

In this paper, we consider a discrete version of the functionals
in~\eqref{DIRI:2} and~\eqref{DIRI:p} in which the gradient is replaced by an oscillation term
in a ball of fixed radius. On the one hand, this new functional retains the property of 
attaining minimal value on constant functions, hence oscillatory functions cause an increasing of the energy values. On the other hand, this new type of functionals is nonlocal, since any modification
of the function at a given point influences the energy density in a fixed ball.
Differently than other kinds of nonlocal functions studied in the literature, the one that we study
here is not scale invariant,
since the radius of the ball on which the oscillation is computed provides a natural threshold of relevant magnitudes.\medskip

More precisely, the mathematical framework in which we work is the following.
For any function $u\in L^1_{\rm loc}(\R^n)$, $x\in\R^n$ and  $r>0$, we 
define  the oscillation of~$u$ in~$B_r(x)$ as 
$$ \osc_{B_r(x)} u:= \sup_{B_r(x)} u-\inf_{B_r(x)} u.$$
It can be checked by using the definition that a triangular inequality holds:
namely, for all $v,u\in  L^1_{\rm loc}(\R^n)$  and $\lambda, \mu\geq 0$, 
$$ \osc_{B_r(x)} (\lambda u+\mu v)\leq  \lambda \osc_{B_r(x)} (u)+\mu \osc_{B_r(x)} (v).$$
 
Given~$p\ge1$
and an open set~$\Omega\subseteq\R^n$, we introduce the functional
\begin{equation}
\label{E r} {\mathcal{E}}_{r, p}(u, \Omega):=
\int_\Omega \left( \osc_{B_r(x)} u \right)^p\,dx,\end{equation}
which we will denote as the  $p$-oscillation functional. 
This functional is $p$-homogeneous, and it is also convex, due to the triangular
inequality and the
convexity of the map~$[0,+\infty)\ni r\mapsto r^p$. 
Therefore it is   lower semicontinuous in $L^1_{\rm{loc}}$,
see e.g.~\cite{MR3023439}.   
Moreover, we observe that  for convex functionals, weak and strong
lower semicontinuity coincide (see e.g. Theorem~9.1 in~\cite{MR2798533}).

Furthermore,
we notice that if $u$ is not locally bounded, then~${\mathcal{E}}_{r, p}(u, \Omega)=+\infty$. 
\medskip

When~$p=1$, this functional can be interpreted as a discrete version
(at scale~$r$) of the total variation functional (see~\cite{MR1857292}). 
Indeed, it can be proved 
that~$\mathcal{E}_{r, 1}(u, \Omega)$ $\Gamma$-converges as~$r\to 0$
to
$$
TV(u, \Omega):=\left\{\begin{matrix}
\displaystyle \int_{\Omega}|\nabla u(x)|\,dx, &{\mbox{ if }} u\in BV(\Omega),\\
+\infty, &{\mbox{ if }}u\not\in BV(\Omega),
\end{matrix}\right.$$
see~\cite[Proposition 3.5]{MR2655948}.
\medskip

We introduce the  definition of minimizers for the functionals ${\mathcal{E}}_{r,p}$, 
in which competitors are fixed in a neighborhood
of width~$r$ of the boundary. The reason for this choice in the definition of minimizers
is due to the fact that the scale $r$ associated to the functional 
has to be taken into account in order not
to trivialize the notion of Class~A minimizers (see the forthcoming
Proposition~\ref{CLASSIFICATION:FAC}). 

We start with some preliminary definitions.
Given~$r>0$, we let
\begin{equation}
\label{omegar}
\begin{split} & \Omega\oplus B_r:=
\bigcup_{x\in \Omega} B_r(x)=(\partial \Omega\oplus B_r)\cup \Omega
=(\partial \Omega\oplus B_r)\cup(\Omega\ominus B_r),\\
{\mbox{where}}\qquad  &
\Omega\ominus B_r:=
\Omega\setminus \left(\bigcup_{x\in\partial \Omega} B_r(x)\right)=\Omega\setminus\big(
(\partial \Omega)\oplus B_r\big).\end{split}\end{equation}

\begin{definition}[Minimizers and Class~A minimizers] \label{DEFIN2} 
Let $\Omega$ be a  open bounded set in $\R^n$. 
We say that~$u\in L^1(\Omega )$ is
a minimizer in~$\Omega$ for~${\mathcal{E}}_{r,p}$ if
$$ {\mathcal{E}}_{r,p}(u, \Omega)\le {\mathcal{E}}_{r,p}(u+\varphi, \Omega)$$
for any~$\varphi\in L^1(\Omega)$
with~$\varphi=0$ in~$\Omega\setminus (\Omega\ominus B_r)$, where  $\Omega\ominus B_r$ is defined in \eqref{omegar}. 

Also, we say that~$u\in L^1_{\rm loc}(\R^n)$ is
a Class~A minimizer if it is a minimizer in any ball of~$\R^n$.
\end{definition} 

In this paper we are interested in the analysis of the main properties of such minimizers.
We start providing in Section~\ref{compas} a compactness result, which we can state as follows:

\begin{proposition}\label{prominfunzioni}Let $\Omega$ be a
open bounded set in $\R^n$, $p\geq 1$ and $u_k$ be a sequence of minimizers
of~${\mathcal{E}}_{r,p}(\cdot, \Omega)$ such that~$u_k\to u$ in~$L^1(\Omega)$. 
Then, $u$ is a minimizer of ${\mathcal{E}}_{r,p}(\cdot , \Omega\ominus B_r)$.

In particular,
if $u_k$ is a sequence of Class~A minimizers
such that $u_k\to u$ in~$L^1_{\rm loc}(\R^n)$, then $u$ is a Class~A minimizer. 
\end{proposition} 

Then, in Section~\ref{mincos} 
we analyze the relation between the functional~$\mathcal{E}_{r,1}$
and a nonlocal perimeter functional.
The framework in which we work goes as follows:
if $E\subseteq \R^n$ is a measurable set, then we denote 
\begin{equation}\label{per}\Per_r(E, \Omega):=  \frac{1}{2r}{\mathcal{E}}_{r, 1}(\chi_E, \Omega)=\frac{1}{2r}{\mathcal{L}}^n \Big(
\big((\partial E)\oplus B_r\big) 
\cap\Omega\Big)\end{equation}
where $(\partial E)\oplus B_r$ is defined in \eqref{omegar}. 

The definition of~$\Per_r$ is inspired by
the classical Minkowski content (which would be recovered in
the limit, see e.g.~\cite{MR2655948, MR3187918}).
In particular, for 
sets with compact and $(n-1)$-rectifiable boundaries, the functional in~\eqref{per}
may be seen as a nonlocal approximation of the classical perimeter functional,
in the sense that
\begin{equation*} \lim_{r\searrow0}  \Per_r(E)=
{\mathcal{H}}^{n-1} (\partial E).\end{equation*}
Then, we point out the following result:

\begin{theorem}\label{EQUIVALENCE}
If the  function~$u$ is a minimizer 
of~${\mathcal{E}}_{r,1}(\cdot, \Omega)$
then for a.e.~$s\in\R$ the level set~$\{u>s\}$
is a minimizer  for~$\Per_r$ in~$\Omega\ominus B_r$.
Viceversa, if for a.e.~$s\in\R$ the level set~$\{u>s\}$
is a minimizer  for~$\Per_r$ in~$\Omega$ then $u$ is a minimizer 
of~${\mathcal{E}}_{r,1}(\cdot, \Omega)$. 
\end{theorem}

We provide additional results on~$\Per_r$
in~\cite{cdnv}. See~\cite{MR2728706, MR2655948, MR3187918, MR3023439, MR3401008}
for a number of related problems and results.
\medskip

One of the main results of this paper is about
the existence of solutions to the Dirichlet problem:

\begin{theorem}\label{EXISTENCE}
Let $\Omega\subseteq\R^n$ be a bounded open set, and $u_o\in L^\infty(\Omega\oplus B_r)$.
Then, there exists~$u\in L^\infty(\Omega\oplus B_r)$ with~$u=u_o$
in~$(\Omega\oplus B_r)\setminus\Omega$  and $\|u\|_{L^\infty(\Omega\oplus B_r)}\leq \|u_o\|_{L^\infty(\Omega\oplus B_r)}$ such that~$
{\mathcal{E}}_{r,p}(u,\Omega)\le {\mathcal{E}}_{r,p}(v,\Omega )$
for any~$v\in L^1_{\rm loc}(\Omega\oplus B_r)$ with~$v=u_o$
in~$(\Omega\oplus B_r)\setminus\Omega$.

Finally, if~$n=1$ and~$u_o\in L^\infty(\Omega\oplus B_r)$ is monotone, there exists a minimizer~$u$ that is also monotone.
\end{theorem} 

This result is proved in Section~\ref{dirs}. 
Finally, in Section~\ref{rigs}
we provide a description for
Class~A minimizers in dimension~1.
More precisely, we collect the results that we obtain in the following statement:

\begin{theorem} \label{THM3}
Let $n=1$. Then, the following holds: 
\begin{enumerate}
\item If ~$u $ is
a Class~A minimizer for 
the functional~$\mathcal{E}_{r,p}$ for some $p\geq 1$,  then $u$ is monotone.
\item Every monotone function is a Class~A minimizer for $\mathcal{E}_{r,1}$.
\item Every  monotone function $u$  such that  $u(x)=C x+ \phi(x)$ for 
some $C\in\R$ and $\phi\in L^1_{\rm loc}(\R)$ which is $2r$-periodic
is a Class~A minimizer for $\mathcal{E}_{r,p}$ for any $p> 1$.
\item If ~$u$ is a Class~A minimizer for $\mathcal{E}_{r,p}$
for some~$p>1$ and $u$ is strictly monotone, then there exist~$C\neq 0$
and~$\phi\in L^1_{\rm loc}(\R)$ which is $2r$-periodic
such that~$u(x)= Cx+\phi(x)$.  \end{enumerate}
\end{theorem} 

A related  problem which is left open is about the validity of rigidity results
for Class~A
minimizers in dimension greater than~1.
In particular, it could be interesting to 
study an analogous of the Bernstein problem for the~$1$-oscillation functional,
in analogy with the classical total variation functional. 
For $p>1$, rigidity type results should be in analogy with classical
Liouville type theorems for $p$-Dirichlet functionals. 

\subsubsection*{Notation}  
In the $\sup$ and $\inf$ notation, we mean
the ``essential supremum and infimum'' of the function (i.e.,
sets of null measure are neglected).  Moreover we shall identify a set~$E\subseteq \R^n$ with its points of density one and 
$\partial E$ with the topological boundary of the set of points of density one. 
Finally for any  $u:I\subset \R \to \R$ monotone function,
we will always identify $u$ with its   right continuous representative. 

\section{Compactness of minimizers}\label{compas}

Here we prove the
compactness result on the minimizers of the oscillation functional stated in Proposition~\ref{prominfunzioni}.

\begin{proof}[Proof of Proposition~\ref{prominfunzioni}]
Let $\Omega':=\Omega\ominus B_r$ and~$\varphi$ such 
that ${\rm{supp}} \varphi = \Omega' \ominus B_r$,
and we claim that 
\begin{equation}\label{iegergregerj}
\mathcal{E}_{r,p}(u+\varphi, \Omega')\geq
\mathcal{E}_{r,p}(u, \Omega').\end{equation}
For this, we define $u_k^*:=(u-u_k)\chi_{\Omega'} +u_k$. 
Then we observe that for a.e. $x\in \Omega$ and for all $p\geq 1$
\begin{equation}\begin{split} \label{o1}
& \Big(\osc_{B_r(x)}u_k^*\Big)^p\leq
\max\left\{\Big(\osc_{B_r(x)}u_k\Big)^p, \Big(
\osc_{B_r(x)}u\Big)^p\right\}\\ \text{ and } \quad &
\Big(\osc_{B_r(x)}u \Big)^p\leq \liminf_k \Big(\osc_{B_r(x)}u_k\Big)^p.
\end{split}\end{equation}
Using \eqref{o1}, we compute 
\begin{eqnarray*}
\int_{\Omega} (\osc_{B_r(x)}u_k^*)^p \,dx
&\leq &\int_{\Omega}( \osc_{B_r(x)}u_k)^p\, dx
+\int_{\Omega} \max(0, ( \osc_{B_r(x)}u)^p-(\osc_{B_r(x)}u_k)^p)\, dx\\&=& 
\int_{\Omega} (\osc_{B_r(x)}u_k)^p\, dx+\omega_k,
\end{eqnarray*} where 
\begin{equation}\label{iegergregerj:2}
{\mbox{$\omega_k\to 0$ as~$k\to+\infty$.}}\end{equation} 
Therefore, we get, by construction and using the minimality of $u_k$, 
\begin{align*}
\mathcal{E}_{r,p}(u+\varphi, \Omega')- \mathcal{E}_{r,p}(u, \Omega') =\mathcal{E}_{r,p}(u_k^*+\varphi, \Omega')-\mathcal{E}_{r,p}(u_k^*, \Omega')
\\\geq \mathcal{E}_{r,p}(u_k^*+\varphi, \Omega')-
\mathcal{E}_{r,p}(u_k, \Omega')-\omega_k\geq -\omega_k. \end{align*}
As a consequence, sending~$k\to+\infty$ and recalling~\eqref{iegergregerj:2},
we obtain~\eqref{iegergregerj}.
\end{proof} 

\section{Relation with the Minkowski perimeter}\label{mincos}

In this section we discuss the relation between
the $p$-Dirichlet functional in~\eqref{E r} and the
Minkowski perimeter
in~\eqref{per}. Namely, we have the  following generalized coarea formula
which relates the functional~$\mathcal{E}_{r,1}$ with the functional~$\Per_r$
(see formulas~(4.3) and~(5.7)
in~\cite{MR3401008} for similar formulas in very related contexts).

\begin{lemma}\label{COAREA:L}
It holds that
\begin{equation}\label{COAREA}
\int_\Omega \osc_{B_r(x)} u\,dx=2r\,\int_{-\infty}^{+\infty}
\Per_{r}(\{u>s\}, \Omega) \,ds.\end{equation}
\end{lemma}

The coarea formula and the previous Proposition \ref{prominfunzioni} provide a link between  
local minimizers   of~${\mathcal{E}}_{r,1}(\cdot, \Omega)$
and the  local minimization of~$\Per_r$ in~$\Omega$ of the level sets,
according to Theorem~\ref{EQUIVALENCE} that we now prove.

\begin{proof}[Proof of Theorem~\ref{EQUIVALENCE}]
In all the proof, we will take~$v$ to be equal to~$u$ 
outside~$\Omega\ominus B_r$, i.e.~$v=u+\phi$, 
with~$\phi$ vanishing
outside~$\Omega\ominus B_r$.

First, we assume that
for a.e.~$s\in\R$ the level set~$\{u>s\}$
is a minimizer  for~$\Per_r$ in~$\Omega$.
Then~$\Per_r\big(\{u>s\},\Omega)\le\Per_r\big(\{v>s\},\Omega)$
for a.e.~$s\in\R$, which combined with the coarea formula in~\eqref{COAREA}
gives that
$$ \int_\Omega \osc_{B_r(x)} u\,dx\le
\int_\Omega \osc_{B_r(x)} v\,dx.$$
This shows that~$u$ is a local minimizer of~${\mathcal{E}}_{r,1}(\cdot, \Omega)$, as desired.

Viceversa, assume now that~$u$ is a local minimizer 
of~${\mathcal{E}}_{r,1}(\cdot, \Omega)$.
Given~$t\in\R$ and~$\lambda>0$, we define
\begin{equation}\label{LULAMB} 
u_{\lambda,s}(x):=\frac12+\max\left\{\min\left\{
\lambda\big(u(x)-s\big),\frac12
\right\}
,\,-\frac12
\right\}.\end{equation}
We claim that
\begin{equation}\label{LU}
{\mbox{$u_{\lambda,s}$ is a   minimizer of~${\mathcal{E}}_{r,1}(\cdot, \Omega)$. }}
\end{equation}
To prove this,
we need to combine
different ideas appearing in the literature in different contexts.
On the one hand, arguing as in
Proposition~3.2 of~\cite{MR2655948},
one sees that the procedure of taking~$\min$ and~$\max$ (as in~\eqref{LULAMB})
makes the energy decrease.
On the other hand, this procedure in general changes the boundary
data hence the minimization in the appropriate class
may get lost (to picture this phenomenon, one can think at the
one dimensional case in which~$u_1(x)=x$ and~$u_2(x)=-x$
may have minimal properties, but the energy of~$\max\{u_1(x),u_2(x)\}=|x|$
may be lowered by horizontal cuts).

Hence, to overcome this difficulty, we will adopt a strategy developed
in Lemma~3.5 of~\cite{MR1930621} to consider specifically the horizontal cuts.
To this end,
we first notice that, for any constant~$c\in\R$,
\begin{equation} \label{OC COS}
\osc_{B_r(x)}u = \osc_{B_r(x)}\min\{u,c\} +
\osc_{B_r(x)}\max\{u,c\}.
\end{equation}
This fact is  a direct consequence of the definition. 
%

Now, for any~$\phi$ supported in $\Omega\ominus B_r$, 
using~\eqref{OC COS} and the minimality of~$u$, we find that
\begin{equation}\label{MI E MA0}
\begin{split}
& \int_\Omega \osc_{B_r(x)}\min\{u,c\} \,dx+\int_\Omega
\osc_{B_r(x)}\max\{u,c\}\,dx\\&\qquad=
\int_\Omega
\osc_{B_r(x)}u \,dx\le\int_\Omega
\osc_{B_r(x)}(u+\phi) \,dx\\&\qquad=\int_\Omega
\osc_{B_r(x)}(u+c+\phi) \,dx
=\int_\Omega\osc_{B_r(x)}\big(\min\{u,c\}+\max\{u,c\}+\phi\big) \,dx
.\end{split}\end{equation}
We also observe that by the triangular inequality
\begin{equation}\label{MI E MA1}
 \osc_{B_r(x)}\big(\min\{u,c\}+\max\{u,c\}+\phi\big) 
\leq \osc_{B_r(x)}\min\{u,c\}+
\osc_{B_r(x)}\big(\max\{u,c\}+\phi\big).\end{equation}
In a similar way, we see that
\begin{equation}\label{MI E MA2}
\osc_{B_r(x)}\big(\min\{u,c\}+\max\{u,c\}+\phi\big)\le
\osc_{B_r(x)}\max\{u,c\}+
\osc_{B_r(x)}\big(\min\{u,c\}+\phi\big)
.\end{equation}
Inserting~\eqref{MI E MA1} into~\eqref{MI E MA0}
and simplifying one term, we obtain that
\begin{equation}\label{MI E MA3}
\int_\Omega
\osc_{B_r(x)}\max\{u,c\}\,dx\le
\int_\Omega
\osc_{B_r(x)}\big(\max\{u,c\}+\phi\big).
\end{equation}
Similarly, plugging~\eqref{MI E MA2} into~\eqref{MI E MA0}
and simplifying one term, we see that
\begin{equation}\label{MI E MA4}
\int_\Omega
\osc_{B_r(x)}\min\{u,c\}\,dx\le
\int_\Omega
\osc_{B_r(x)}\big(\min\{u,c\}+\phi\big).
\end{equation}
{F}rom~\eqref{MI E MA3}, we find that~$\max\{u,c\}$
is a minimizer with respect to the perturbation~$\phi$, while
from~\eqref{MI E MA4}
it follows that~$\min\{u,c\}$
is also a minimizer with this perturbation.
These considerations and~\eqref{LULAMB}
imply~\eqref{LU}, as desired.

Now we observe that, for a.e.~$s\in\R$, we have that
\begin{equation}\label{LEB}
{\mbox{$\{u=s\}$
has zero Lebesgue measure,}}\end{equation}
otherwise the disjoint union of these sets would have locally infinite
Lebesgue measure, and therefore
\begin{equation}\label{PASS:LI1}
{\mbox{$
u_{\lambda,s}$ converges to $\chi_{\{u>s\}}$
in $L^1_{\rm loc}(\R^n)$, as~$\lambda\to+\infty$.}}
\end{equation}
So, thanks to Proposition~\ref{prominfunzioni}, we conclude from~\eqref{LU}
that~$\chi_{\{u>s\}}$ is a local minimizer in~$\Omega\ominus B_r$.
\end{proof}

As a consequence of
Theorem~\ref{EQUIVALENCE}, we obtain the next proposition,
which explains why in the definition of minimizer in Definition \ref{DEFIN2}
we allow competitors
in a neighborhood of width~$r$ of the boundary. 

\begin{proposition}\label{CLASSIFICATION:FAC}
Let  $u\in L^1_{\rm loc}(\R^n)$ such that for every ball $B$ 
$$ {\mathcal{E}}_{r,1}(u,B )\le {\mathcal{E}}_{r, 1}(u+\varphi, B)$$
for any~$\varphi\in L^1 (\R^n)$
with~$\varphi=0$ in~$\R^n\setminus B$.

Then~$u$ is necessarily constant.
\end{proposition}
\begin{proof} \ 
By Theorem~\ref{EQUIVALENCE}
it holds that, for a.e.~$s\in\R$, $E_s=\{u>s\}$  satisfies the property that 
for any measurable set~$F\subseteq\R^n$
with~$F\setminus B=E_s\setminus B$ it holds that
\[ \Per_r(E_s,B)\le\Per_r(F,B).\]
So, by \cite[Proposition 1.3]{cdnv}, 
either~$E_s=\varnothing$ or~$E_s=\R^n$.
As a consequence~$u$ is constant, as desired.
\end{proof} 

\section{Existence for the Dirichlet problem} \label{dirs}

We provide here
the proof of Theorem~\ref{EXISTENCE} about
Dirichlet problem for the functional~${\mathcal{E}}_{r,p}$
and the one-dimensional monotonicity property.

\begin{proof}[Proof of Theorem \ref{EXISTENCE}]
The existence
result is a straightforward application of the  direct method, recalling that the functional is weak lower semicontinuous. 
Also, since cutting a function at the level~$\pm L$
decreases its oscillation,
we can reduce our competitors to bounded functions with $L^\infty$ norm bounded by $\|u_o\|_{L^\infty(\Omega\oplus B_r)}$.

\smallskip

Suppose now that~$n=1$ and $u_o$ is nondecreasing, and let~$\Omega=(a,b)$, for some~$b>a\in\R$.
First of all, we show that a minimizer would not overcome the value of~$u_o(b)$
inside~$(a,b)$. Namely, 
given any~$u:(a-r,b+r)\to\R$, with $u=u_o$ for $x\in (a-r,a)\cup (b, b+r)$, we set
\begin{equation}\label{5566} \vartheta_u(x):=\left\{
\begin{matrix}
\min\{ u_o(b),\, u(x)\} & {\mbox{ if }} x\in(a,b)\\
u(x)(=u_o(x))& {\mbox{ if }}x\in (a-r, a]\cup [b,b+r).
\end{matrix}
\right. \end{equation}
By construction $u(x)\geq \vartheta_u(x)$ for all $x\in (a-r, b+r)$. 
We claim that, for any interval~$I\subseteq(a-r,b+r)$,
\begin{equation}\label{meg}
\osc_I \vartheta_u\le\osc_I u.
\end{equation}
 
It is easy to check using definitions  that if $x,y\in (a, b)$, then there holds \begin{equation}\label{4ca}
|\vartheta_u(x)-\vartheta_u(y)|\le| u(x)-u(y)|.
\end{equation}  
Indeed if either  both $u(x), u(y)\geq u_o(b)$, or $u(x), u(y)\leq u_o(b)$ there is nothing to prove. If $u(x)\geq u_o(b)>u(y)$, then $\vartheta_u(x)= u_o(b)$ and $\vartheta_u(y)=u(y)$, so 
\[|\vartheta_u(x)-\vartheta_u(y)|= u_o(b)-u(y)\leq u(x)-u(y)=|u(x)-u(y)|.\]
Therefore we get that for all $I\subseteq (a,b)$, then \eqref{meg} holds. 

Assume now that for some $\eps>0$ small, either $(a-\eps, a+\eps)\subseteq I$ or $(b-\eps, b+\eps)\subseteq I$. Then
we observe that $\inf_{I} u= \inf_I \vartheta_u$, so recalling that $u\geq \vartheta_u$, we conclude again that \eqref{meg} holds.

Notice that, as a consequence of~\eqref{meg},
\begin{equation}\label{SUBS:0}
{\mbox{if $u$ is a minimizer, then so is $\vartheta_u$}}.
\end{equation}

Let  $u:(a-r,b+r)\to\R$ be a minimizer. So,  $u=u_o$ in~$(a-r,a]\cup[b,b+r)$. Moreover,  eventually replacing $u$ with $\vartheta_u$, we can assume that 
$u\leq u_o(b)$ in $(a,b)$.

We denote by~$\eta_u$ the nondecreasing envelope of $u$, defined as 
\begin{equation}\label{DETA} \eta_u(x):= \sup_{\tau\in (a-r,x]} u(\tau).\end{equation}
By definition $\eta_u\geq u$, therefore $\inf_I \eta_u\geq \inf_I u$, for all $I\subseteq (a-r, b+r)$. 
Moreover, by monotonicity of $u_o$ and  since  $u\leq u_o(b)$ in $(a,b)$,  we get that $\eta_u=u_o$  in~$(a-r,a]\cup [b, b+r)$.   

So $\eta_u$ is a competitor for the minimizer~$u$. 
We claim that~$\eta_u$ is also a minimizer, namely
\begin{equation}\label{vetau}
{\mathcal{E}}_{r,p }(u, (a,b))\ge
{\mathcal{E}}_{r,p}(\eta_u, (a,b)).
\end{equation}
To this end, we show that, for any~$x\in(a,b)$,
\begin{equation}\label{osc mo}
\osc_{(x-r,x+r)} \eta_u   \le\osc_{(x-r,x+r)} u.
\end{equation}
First of all we observe that for almost every $x$, there holds
$\osc_{(x-r,x+r)} \eta_u =\eta_u(x+r)-\eta_u(x-r)$. 

We may suppose that $
\eta_u(x-r)<\eta_u(x+r),
$
otherwise we would have that $ \osc_{(x-r,x+r)} \eta_u=0\le\osc_{(x-r,x+r)} u,$ 
as desired.

We observe that 
\begin{equation}\label{019284}
{\mbox{for any $y\in(a-r,x-r]$, }}
\eta_u(x+r)>\eta_u(x-r)\ge \eta_u(y) \ge u(y).
\end{equation}
We claim that \begin{equation}\label{8uUUA}
\sup_{(x-r,x+r)} \eta_u=\sup_{(x-r,x+r)} u.
\end{equation}
To check this, let us assume, for a contradiction, that
$$ \sup_{(x-r,x+r)} \eta_u>\sup_{(x-r,x+r)} u.$$
Then, for any~$y\in(x-r,x+r)$, we have that
$u(y)\le \eta_u(x+r)-\alpha$, for some~$\alpha>0$.
Up to changing~$\alpha>0$, this holds true also for any~$y\in(a-r,x+r)$,
thanks to~\eqref{019284}.
Consequently, by~\eqref{DETA},
$$ \eta_u(x+r)= \sup_{y\in (a-r,x+r]} u(y) \le \eta_u(x+r)-\alpha,$$
which is of course a contradiction, which establishes~\eqref{8uUUA}.
This gives also ~\eqref{osc mo}, recalling that $\eta_u\geq u$. 

Then, from~\eqref{osc mo} we deduce~\eqref{vetau}. Hence, $\eta_u$ is a minimizer,
and it is monotone, as desired.
\end{proof}

\begin{remark}\label{REM:NP}
We stress that the minimizer given by Theorem~\ref{EXISTENCE}
is not necessarily unique (not even when~$p>1$).
Also, when~$n=1$,
it is not necessarily monotone (not even when~$u_o$ is monotone).
Finally, it is not necessarily continuous (not even when~$u_o$ is analytic).

We consider for example, $\Omega:=(-1,1)\subset\R$, $r:=3$ and $u_o(x):=x$.
Notice that, for any~$x\in(-1,1)$, it holds that~$x-3 <-1$
and~$x+3>1$. Accordingly, if~$v$
coincides with~$u_o$
outside~$(-1,1)$, then
\begin{eqnarray*}
 \sup_{(x-3,x+3)} v\ge u_o(x+3)=x+3\qquad 
\inf_{(x-3,x+3)} v\le u_o(x-3)=x-3,
\end{eqnarray*}
which implies that
$ \osc_{(x-3,x+3)} v\ge(x+3)-(x-3)=6,$
and thus
$ {\mathcal{E}}_{3,p }(v, (-1,1))
\ge 2\cdot 6^p.$
This says that any function~$u$ that
coincides with~$u_o$
outside~$(-1,1)$ and satisfies
$$ \sup_{(-1,1)} u\leq 1 \quad{\mbox{ and }}\quad
\inf_{(-1,1)}u\geq -1$$
is a minimizer in the sense of
Theorem~\ref{EXISTENCE}.
\end{remark}

\begin{remark}\label{NOZ}
It is interesting to point out that the ``inverse problem''
in Theorem~\ref{EXISTENCE} is not well posed, in the sense that
a minimizer~$u$ does not determine uniquely the datum~$u_o$.
For instance, while the null functions is obviously a minimizer
for null data, it may also be a minimizer for nontrivial data.

Assume e.g. that $n=1$ and~$\Omega=(a,b)$ for some~$b>a$, and  $$ u_o(x):=\left\{
\begin{matrix}
1 & {\mbox{ if }}x\in\left( a-r,a-\frac{r}{2}\right],\\
0 & {\mbox{ if }}x\in\left(a-\frac{r}{2},a\right]\cup
[b,b+r).
\end{matrix}
\right.$$

In this case the null function~$u$ in~$(a,b)$,
extended to~$u_o$ in~$(a-r,a]\cup[b,b+r)$ is a minimizer
according to Theorem~\ref{EXISTENCE}.
For this, we observe that if~$v=u_o$
in~$(a-r,a]\cup[b,b+r)$ and~$x\in \left( a,a+\frac{r}{2}\right)$,
it holds that~$\left\{ a-\frac{r}{2}\right\}\in (x-r,x+r)$,
and therefore ``$v$ sees the jump of~$u_o$ in such interval'', that is,
for any~$x\in \left( a,a+\frac{r}{2}\right)$,
$$ \osc_{(x-r,x+r)} v\ge 1.$$
This implies that
\begin{equation}\label{DASOTT}
{\mathcal{E}}_{r,p }(v, (a,b))\ge \int_a^{a+\frac{r}2} 
\left(\osc_{(x-r,x+r)} v\right)^p
\,dx\ge \frac{r}{2}.\end{equation}
Now, the null function~$u$
extended to~$u_o$ in~$(a-r,a]\cup[b,b+r)$ satisfies,
for any~$x\in \left( a,a+\frac{r}{2}\right)$,
$ \osc_{(x-r,x+r)} u = 1$ 
and, for any~$x\in \left( a+\frac{r}{2},b\right)$, $ \osc_{(x-r,x+r)} u = 0.$
Consequently, we have that
$$ {\mathcal{E}}_{r,p }(u, (a,b))= \int_a^{a+\frac{r}2}
\left(\osc_{(x-r,x+r)} v\right)^p
\,dx=\frac{r}{2}.$$
By comparing this with~\eqref{DASOTT}, we conclude that~$u$
is a minimizer, as desired.
\end{remark}

\section{Rigidity properties of minimizers in dimension $1$} \label{rigs}

In this section, we provide some rigidity results about
Class~A minimizers in dimension~1 for
the functional~$\mathcal{E}_{r,p}$,
both in the cases~$p=1$ and~$p>1$.

We start with the following result:

\begin{proposition}\label{minmon}
Let~$u\in L^1_{\rm loc}(\R)$ be 
a Class~A minimizer for 
the functional~$\mathcal{E}_{r,p}$, for some~$p\geq 1$.
Then~$u$ is monotone.
\end{proposition} 

\begin{proof} We suppose by contradiction that~$u$ is not monotone.    
First of all, we observe that $u$ has to be locally bounded. 

For any $x\in\R$, we denote by 
\[\underline{u}(x) :=\sup_{\eps>0}\inf_{(x-\eps, x+\eps)} u
\leq \overline{u}(x) :=\inf_{\eps>0}\sup_{(x-\eps, x+\eps)} u.\]
Note that, for any
Lebesgue point~$x$, we get that $\underline{u}(x)\leq u(x)\leq \overline{u}(x)$.

Moreover, for any Lebesgue point~$x$ of~$u$, it holds that 
\begin{equation}\label{leb} 
\inf_{(x-\eps, x+\eps)}u\leq u(x)\leq \sup_{(x-\eps, x+\eps)}u.\end{equation}
To check~\eqref{leb}, we argue by contradiction and we suppose, for instance, that 
there exists~$\delta>0$ such that 
$$ \sup_{(x-\eps, x+\eps)}u-u(x)=-\delta<0.$$
Thus, for every~$\tilde\eps\in(0,\eps)$,
$$\sup_{(x-\tilde\eps, x+\tilde\eps)}u\leq  u(x)-\delta.$$ 
From this we deduce that 
$$\lim_{\tilde\eps\to 0}\frac{1}{\tilde\eps} \int_{x-\tilde\eps}^{x+\tilde\eps}
|u(y)-u(x)|\,dy\geq \delta>0,$$
which is in contradiction with the fact that $x$ is a Lebesgue point for $u$.
This proves~\eqref{leb}.

Now, since~$u$ is not monotone, we can suppose that there exist~$a<b$ such that
\begin{equation} 
\label{cond} 
[\underline{u}(a), \overline{u}(a)]\cap [\underline{u}(b), \overline{u}(b)]\not= 
\varnothing\qquad \text{and } \qquad \osc_{(a,b)} u>0.\end{equation}
In virtue of~\eqref{cond}, we let 
\[c\in [\underline{u}(a), \overline{u}(a)]\cap [\underline{u}(b), \overline{u}(b)]\]
and define the function  
\begin{equation}\label{utilde} \tilde u(x):=\begin{cases}u(x) & {\mbox{ if }} x<a {\mbox{ and }}
x>b,\\ 
c & {\mbox{ if }} x\in [a,b].\end{cases}\end{equation} 
Then, we get that 
\[\osc_{B_r(x)} \tilde u\leq  \osc_{B_r(x)}u \qquad {\mbox{ for any }} x\in  \R^n.\]
Now, we observe that, 
if $a,b$ are given by~\eqref{cond},
then necessarily 
\begin{equation}\label{irthy568}
b-a\leq 2r.\end{equation}
Indeed, if on the contrary $b-a>2r$,
for all $x\in (a+r, b-r)$, we have that~$
\osc_{B_r(x)} \tilde u=0$, thanks to~\eqref{utilde}.
On the other hand, since~$\osc_{(a,b)} u>0$ (recall~\eqref{cond}), 
there exists a set~$E\subseteq (a+r, b-r)$ of
positive measure such that~$\osc_{B_r(x)}  u>0$ for all~$x\in E$.
Hence, we would get that 
$$\mathcal{E}_{r,p}(\tilde u, (a+r, b-r))<
\mathcal{E}_{r,p}(u, (a+r, b-r)),$$ 
which contradicts the minimality of~$u$.
This proves~\eqref{irthy568}.

Now, we fix~$a$ and~$b$ to be the maximal ones for which~\eqref{cond}
holds true (namely, we suppose that 
it is not possible to find another couple~$a'$ and~$b'$
such that~$a'<a$, $b'>b$ and~\eqref{cond} is satisfied). 
In this case, we can show that   
\begin{equation}\label{uno}
{\mbox{either }}\;
\underline{u}(x)>c \,{\mbox{ for any }} x<a\; {\mbox{ and}}\;
\overline{u}(x)<c 
\,{\mbox{ for any }} x>b,
\end{equation}
\begin{equation}\label{uno1}
{\mbox{or }}\;\overline{u}(x)<c \,{\mbox{ for any }} 
x<a\; {\mbox{ and }}\; \underline{u}(x)>c \,{\mbox{ for any }} x>b.
\end{equation} 
Indeed if we were not in this situation, then the maximality of the
couple~$a$, $b$ would be contradicted.

{F}rom now on, we suppose that~\eqref{uno} is satisfied
(being the case~\eqref{uno1} completely analogous).
Since~$\osc_{(a,b)} u>0$ (in virtue of~\eqref{cond}),
we get that there exists an interval~$(\alpha,\beta)\subset \subset(a,b)$
such that either~$\sup_{(\alpha, \beta)} u> c$
or~$\inf_{(\alpha,\beta)}u<c$.

Assume for instance that~$\sup_{(\alpha, \beta)} u> c$
and fix any $x\in (a+r, \alpha+r)$. Then, we have that~$a<x-r<\alpha$,
and so necessarily $x+r>b>\beta$, due to~\eqref{irthy568}.  
In this way, we have that
\[\osc_{B_r(x)} u=\sup_{(x-r, x+r)}u-\inf_{(x-r, x+r)}u\geq \sup_{(\alpha,\beta)} u-\inf_{(x-r, x+r)}u>c-\inf_{(b, x+r)}u=\osc_{B_r(x)} \tilde u.\] 
This implies that
\begin{equation*}
\mathcal{E}_{r,p}(\tilde u, (a+r,\alpha+r))<\mathcal{E}_{r,p}
(u, (a+r,\alpha+r)),
\end{equation*} 
which is in contradiction with
the fact that~$u$ is a Class~A minimizer.
This completes the proof of Proposition~\ref{minmon}.
\end{proof} 

In particular, in the case~$p=1$ we have the following characterization of Class~A
minimizers:
 
\begin{theorem}\label{LAR}
A function~$u\in L^1_{\rm loc}(\R)$ is
a Class~A minimizer for  $\mathcal{E}_{r,1}$
if and only if it is monotone.
\end{theorem}

\begin{proof} 
Assume first that~$u\in L^1_{\rm loc}(\R)$ is
a Class~A minimizer. Then, by Proposition \ref{minmon},
we conclude that~$u$ is monotone. 

Let now assume that~$u$ is a monotone function and
we prove that $u$ is a Class~A minimizer. 
First of all, we observe that if
$v\in L^1_{\rm loc}(\R)$, then, for almost every $x\in\R$, there holds 
\begin{equation}\label{54y85njsfd}
\osc_{(x-r,x+r)} v\ge |v(x+r)-v(x-r)|,\end{equation}
with equality if~$v$ is monotone.
Indeed, arguing as in \eqref{leb},
we get that at all the Lebesgue points~$x$ of~$v$, it holds that 
$$\inf_{(x,x+r)} v\leq v(x)\leq \sup_{(x, x+r)} v,$$
and similarly
$$\inf_{(x-r,x)} v\leq v(x)\leq \sup_{(x-r, x)} v,$$ 
which imply~\eqref{54y85njsfd}.

Now, we fix an interval~$(a,b)\subseteq\R$ with~$b-a>2r$ and we take
a function~$v\in L^1_{\rm loc}(\R)$ which 
coincides with~$u$ outside~$(a+r,b-r)$.
Using~\eqref{54y85njsfd} and making suitable changes
of variables, we obtain that 
\begin{eqnarray*} \int_{a}^{b} \osc_{(x-r,x+r)} v\,dx&\ge &
\int_{a}^{b}\big|
v(x+r)-v(x-r)\big|\,dx \\&\ge&\left| \int_{a}^{b}\big(
v(x+r)-v(x-r)\big)\,dx\right|\\&=&
\left|\int_{b-r}^{b+r} v(y)\,dy-
\int_{a-r}^{a+r} v(y)\,dy\right|
\\&=& \left|\int_{b-r}^{b+r} u(y)\,dy-
\int_{a-r}^{a+r} u(y)\,dy\right|\\ &=&\left|\int_{a}^{b}\big(
u(x+r)-u(x-r)\big)\,dx\right|\\&=& \int_{a}^{b}\big|
u(x+r)-u(x-r)\big|\,dx\\&=&\int_{a}^{b} \osc_{(x-r,x+r)} u\,dx
\end{eqnarray*}
where the last two equalities come from the fact
that $u$ is monotone, and~\eqref{54y85njsfd} has been used
once again in this case. 
This shows that~$u$ is
a Class~A minimizer, and so the proof of Theorem~\ref{LAR}
is completed.
\end{proof}

In the case  $p>1$ we do not have a
complete description of Class~A minimizers, but we can state
the following two results: 

\begin{proposition}\label{LAR2}
Let $p>1$. 
Let $u\in L^1_{\rm loc}(\R)$  be a monotone function such that
$u(x)=Cx+\phi(x)$, for 
some $C\in\R$ and $\phi\in L^1_{\rm loc}(\R)$ which is $2r$-periodic.
Then~$u$
is a Class~A minimizer for $\mathcal{E}_{r,p}$.  
\end{proposition}

\begin{proof}
We observe that, by the definition of~$u$,
for a.e. $x\in\R$, 
\begin{equation}\label{m} \osc_{(x-r,x+r)} u
=|C(x+r)+\phi(x+r)-C(x-r)-\phi(x-r)|=2|C|r.\end{equation} 
Now, we fix an interval~$(a,b)\subseteq\R$ and a function~$
v\in L^1_{\rm loc}(\R)$ which 
coincides with~$u$ outside~$(a+r,b-r)$.
Reasoning as in the proof of Theorem \ref{LAR}, since $u$ is
monotone, one can prove that 
\[ \int_{a}^{b} \osc_{(x-r,x+r)} v\,dx\ge
\int_{a}^{b} \osc_{(x-r,x+r)} u\,dx.\]
Using this and the Jensen inequality, we get that
\begin{eqnarray*}\frac{1}{b-a}
\int_{a}^{b} \left(\osc_{(x-r,x+r)} v\right)^p\,dx &\geq &
\left( \frac{1}{b-a}\int_{a}^{b} \osc_{(x-r,x+r)} v\,dx\right)^p\\ &\geq&
\left( \frac{1}{b-a}\int_{a}^{b} \osc_{(x-r,x+r)} u\,dx\right)^p\\
&=& (2|C|r)^p
,\end{eqnarray*}
where in the last equality we used~\eqref{m}. 
This permits to conclude that
\[ \int_{a}^{b} \left(\osc_{(x-r,x+r)} v\right)^p\,dx\geq (b-a)
(2|C|r)^p=\int_{a}^{b} \left(\osc_{(x-r,x+r)} u\right)^p\,dx,\]
and so $u$ is a Class~A minimizer for $\mathcal{E}_{r,p}$, as desired. 
\end{proof} 

\begin{proposition}\label{LAR3}
Let $p>1$.
Let~$u\in L^1_{\rm loc}(\R)$ be 
a Class~A minimizer for  $\mathcal{E}_{r,p}$.
Suppose that~$u$ is strictly monotone.
Then, there exist~$C\neq 0$ and~$\phi\in L^1_{\rm loc}(\R)$
which is $2r$-periodic, such that $u(x)=Cx+\phi(x)$.
\end{proposition}

\begin{proof}
We suppose that~$u$ is strictly increasing
(being the other case similar). 
We fix $a<b$ such that $b-a>2r$,
and we take a function~$\psi\in C^{\infty}(\R)$
such that~$\psi= 0$ in~$(-\infty, a+r]\cup [b-r, +\infty)$.
Then, for every $\delta\in \R$ such that $u+\delta\psi$ is
still nondecreasing in $(a,b)$, we get that 
$$ \int_{a}^{b}\left( \osc_{(x-r,x+r)} (u+\delta\psi)\right)^p\,dx\ge
\int_{a}^{b}\left( \osc_{(x-r,x+r)}
u\right)^p\,dx,$$
since~$u$ is a Class~A minimizer. Namely, we see that
$$
\int_{a}^{b} \Big(u(x+r)-u(x-r)+\delta(\psi(x+r)-\psi(x-r))\Big)^p\,dx\ge
\int_{a}^{b} \big(u(x+r)-u(x-r)\big)^p\,dx.
$$
This implies that
\[p\int_a^b \Big(u(x+r)-u(x-r)\Big)^{p-1} 
\big(\psi(x+r)-\psi(x-r)\big)\,dx=0. \]
Hence, recalling that~$\psi= 0$ in $(-\infty, a+r]\cup [b-r, +\infty)$,
this gives that
\[\int_{a+r}^{b-r}  \left(u(x)-u(x-2r)\right)^{p-1}\psi(x)\,dx-
\int_{a+r}^{b-r} \left(u(x+2r)-u(x)\right)^{p-1}\psi(x)\,dx=0.\]
As a consequence, using the fact that~$u$ is strictly monotone,
we get the following condition on $u$:
\begin{equation}\label{ire57hf07897}
u(x+2r)-u(x)=u(x)-u(x-2r)\qquad \text{for a.e. }x\in\R. 
\end{equation}
Let now~$C(x):=u(x)-u(x-2r)$.
Then, we have that~$C(x)> 0$ for all $x$, and~$C(x+2kr)=C(x)$
for all $k\in\Z$. Also, from~\eqref{ire57hf07897} we get that,
for every~$k\in\Z$, 
\begin{equation}\label{e} u(x+2kr)=u(x)+k[u(x)-u(x-2r)]=u(x)+k C(x).\end{equation} 
We claim now that 
\begin{equation}\label{ccost}
{\mbox{$C(x)\equiv C$, for some $C>0$.}}\end{equation}
To prove~\eqref{ccost}, we assume on the contrary that there
exist~$x_1$, $x_2$ such that~$|x_1-x_2|<2r$ and~$C(x_1)>C(x_2)$. 
We fix $k_0\in\N$  sufficiently large  such that $C(x_1)>\frac{k+1}{k} C(x_2)$ for all $k\in\N$ such that $k\geq k_0$. 
Then, for all $k\geq k_0$,  using \eqref{e}, and recalling that~$u$
is strictly
monotone and that~$x_2+2(k+1)r>x_1+2kr$,  we get 
\[u(x_1)+kC(x_1)=u(x_1+2kr)<u(x_2+2(k+1)r)=u(x_2)+(k+1)C(x_2)
<u(x_2)+k C(x_1).\]
This implies that~$u(x_1)<u(x_2)$, and therefore~$x_1<x_2$, 
which gives that $C$ is monotone. 
But this is in contradiction with the fact that $C$ is $2r$-periodic,
and so~\eqref{ccost} is proved.

Now we define the function $\phi(x):=u(x)-\frac{C}{2r} x$. 
We have that~$\phi\in L^1_{\rm loc}(\R)$.
Moreover, using \eqref{e}, we check that $\phi$ is a $2r$-periodic
function:
$$ \phi(x+2kr)=u(x+2kr)-\frac{C }{2r}x -\frac{C }{2r}2kr 
= u(x)+kC -\frac{C}{2r}x -Ck= \phi(x). $$
Hence, the proof of Proposition~\ref{LAR3} is complete.
\end{proof} 

With this, we can now summarize the previous results,
thus completing the proof of Theorem~\ref{THM3}.

\begin{proof}[Proof of Theorem~\ref{THM3}]
The claim in~(1) follows from Proposition~\ref{minmon}
whereas the claim in~(2) is a consequence of
Theorem~\ref{LAR}. Furthermore, the claim in~(3)
is warranted by Proposition~\ref{LAR2}
and the one in~(4) follows from Proposition~\ref{LAR3}.\end{proof}

\end{document}